\documentclass[intlim,righttag,10pt]{amsart}
\usepackage{amscd}
\usepackage{amssymb}
\usepackage{stmaryrd}
\usepackage[all]{xy}
\usepackage{hyperref}
\usepackage{url}
\usepackage{xcolor}
\newtheorem{theorem}{Theorem}[section]
 \newtheorem{lemma}[theorem]{Lemma}
 \newtheorem{proposition}[theorem]{Proposition}
 \newtheorem{corollary}[theorem]{Corollary}
 \newtheorem{conjecture}[theorem]{Conjecture}
\newtheorem{fact}[theorem]{Fact}

\theoremstyle{definition}
 \newtheorem{definition}[theorem]{Definition}
 
 \theoremstyle{remark}
 \newtheorem{remark}[theorem]{Remark}

\newcommand{\N}{\mathbf{N}}
\newcommand{\Q}{\mathbf{Q}}

\newcommand{\rig}{\mathrm{rig}}

\newcommand{\Spec}{\mathrm{Spec}}

\newcommand{\Frac}{\mathrm{Frac}}

\newcommand{\ord}{\operatorname{ord}}

\renewcommand\labelenumi{(\roman{enumi})}
\renewcommand\theenumi\labelenumi

\begin{document}

\title[Infiniteness]{On infiniteness of integral overconvergent de Rham--Witt cohomology modulo torsion}
\author[V. Ertl]{Veronika Ertl}
\address[V. Ertl]{Fakult\"at f\"ur Mathematik, Universit\"at Regensburg, Universit\"atsstra\ss{}e 31, 93053 Regensburg, Germany}
\email{veronika.ertl@mathematik.uni-regensburg.de}
\author[A. Shiho]{Atsushi Shiho}
\address[A. Shiho]{Graduate School of Mathematical Sciences, the University of Tokyo, 3-8-1 Komaba, Meguro-ku, Tokyo 153-8914, Japan}
\email{shiho@ms.u-tokyo.ac.jp}  
\date{\today}
\thanks{\copyright 2019 Tohoku Mathematical Journal. All rights reserved under the universal copyright convention.}
\maketitle 

\begin{abstract}
In this article, we give examples of smooth varieties of positive characteristic whose first integral overconvergent 
de Rham--Witt cohomology modulo torsion is not finitely generated over 
the Witt ring of the base field. 

\end{abstract}
\medskip

%%%%%%%%%%%%%
%%%
\section*{Introduction}
%%%
%%%%%%%%%%%%

Let $k$ be a perfect field of characteristic $p>0$, 
let $W := W(k)$ be the ring of $p$-typical Witt vectors of $k$ and 
let $K$ be the fraction field of $W$. 
For a smooth variety $X$ over $k$, its crystalline cohomology 
$H^i_{\mathrm{crys}}(X/W)$ is defined by Berthelot \cite{B_1974} and it is shown 
that, when $X$ is proper, it is finitely generated over $W$. (See \cite{BO_1978} for example.) 
Also, Illusie \cite{I_1979} introduced the notion of de Rham--Witt 
complex $W\Omega_X^{\bullet}$ which is a complex of \'etale sheaves on $X$ and 
proved that its cohomology $H^i(X, W\Omega_X^{\bullet})$ (called the de Rham--Witt 
cohomology) is isomorphic to the crystalline cohomology $H^i_{\mathrm{crys}}(X/W)$. 
Thus the de Rham--Witt cohomology is finitely generated over $W$ when 
$X$ is proper smooth over $k$. 

When $X$ is smooth but not proper, its crystalline cohomology (hence its 
de Rham--Witt cohomology as well) is not necessarily finitely generated. 
To remedy this infiniteness, 
Berthelot \cite{B_1996}, \cite{B_1997} introduced the notion of rigid cohomology $H^i_{\mathrm{rig}}(X/K)$ 
as a corrected variant of crystalline cohomology tensored with $\Q$  
using $p$-adic analytic geometry, and proved that it is finite dimensional over $K$. 
However, rigid cohomology does not a priori have a canonical $W$-lattice. 
So it would be an interesting problem to construct a finitely generated $W$-lattice of 
rigid cohomology which has nice properties. 

For a smooth variety $X$ over $k$, 
Davis--Langer--Zink \cite{DLZ_2011} introduced the overconvergent de Rham--Witt 
complex $W^{\dagger}\Omega_X^{\bullet}$ as a certain subcomplex of 
$W\Omega_X^{\bullet}$ and proved that, when $X$ is quasi-projective, 
its rational cohomology $H^i(X, W^{\dagger}\Omega_X^{\bullet}) \otimes \Q$ 
is isomorphic to the rigid cohomology $H^i_{\mathrm{rig}}(X/K)$. 
Although it is well-known that the integral overconvergent de Rham--Witt 
cohomology $H^i(X, W^{\dagger}\Omega_X^{\bullet})$ can have infinitely generated torsions (e.g. in the case $X = \mathbf{A}^1_k$ and $i=1$), 
one may naively expect that the image 
\[ \overline{H}^i(X, W^{\dagger}\Omega_X^{\bullet}) := 
{\mathrm{Im}}(H^i(X, W^{\dagger}\Omega_X^{\bullet}) \to 
H^i(X, W^{\dagger}\Omega_X^{\bullet}) \otimes \Q), \] 
which we will call the integral overconvergent de Rham--Witt cohomology modulo torsion, 
might give a finitely generated $W$-lattice of the rigid cohomology. 

However, 
various problems seem to arise when one tries to adapt the proofs for finiteness of rigid cohomology in \cite{B_1997}, \cite{M_1997}, \cite{T_1999}, \cite{T_2003}, and \cite{K_2006} to 
the case of $\overline{H}^i(X, W^{\dagger}\Omega_X^{\bullet})$, because 
all of these proofs use homological algebra at some instance which 
is rather delicate when dividing by torsion.
In this article, we give a negative answer to 
the above expectation by giving counterexamples. In particular, 
we prove the following result. 

\begin{theorem}[= a weak form of Corollary \ref{cor}]
For any prime number $p$ and any perfect field $k$ of characteristic $p$, 
there exists an affine smooth curve $X$ over $k$ such that the first integral 
overconvergent de Rham--Witt cohomology modulo torsion 
$\overline{H}^1(X, W^{\dagger}\Omega_X^{\bullet})$ is \textbf{not} finitely generated 
over $W$. 
\end{theorem}

We note that to facilitate the necessary computations, we consider integral Monsky--Washnitzer cohomology, 
where we provide examples  for infiniteness modulo torsion for its first cohomology group.
Then we deduce the above theorem from a comparison isomorphism of integral overconvergent de Rham--Witt cohomology and integral Monsky--Washnitzer cohomology. 
For an example of a higher dimensional or a non-affine smooth variety $X$ with 
infinitely generated $\overline{H}^1(X, W^{\dagger}\Omega_X^{\bullet})$, 
see Theorem \ref{main4} and Remark \ref{main5}.

%%%%%%%%%%%%%%%
\subsection*{Conventions}
%%%%%%%%%%%%%%%

Throughout the paper, $p$ will be a fixed prime number, 
$k$ will be a perfect field of positive characteristic $p>0$, $W := W(k)$ its ring of $p$-typical Witt vectors and $K=\Frac(W(k))$ will be the fraction field of $W$. 
By a variety over $k$ we always mean a separated and integral scheme of finite type 
over $k$.
Let moreover $\nu_p$ denote the $p$-adic valuation.

%%%%%%%%%%%%%%%
\subsection*{Acknowledgement}
%%%%%%%%%%%%%%%

The first-named author's research is partially supported by the DFG grant: SFB 1085 ``Higher invariants''.
The work on this paper started when the first-named author was visiting  Keio University, Yokohama, Japan. 
During this time she was supported by the Alexander~von~Humboldt-Stiftung and the Japan Society for the Promotion of Science as a JSPS International Research Fellow. 
She would like to thank all members of the KiPAS-AGNT group at Keio University  for providing a pleasant working atmosphere.

The second-named author would like to thank Kohei Yahiro for 
explaining the content of the article \cite{DV_2006} in a seminar a few years ago. 
He also would like to thank Andreas Langer for his comment on the first version of 
the article. The second-named author is partially supported by 
JSPS KAKENHI (Grant Numbers 17K05162, 15H02048, 18H03667 and 18H05233). 
Moreover, a revision of the present article was done during the second-named author's 
stay at IMPAN, which was partially supported by the grant 346300 for IMPAN from the Simons Foundation and the matching 2015-2019 Polish MNiSW fund. He would like to 
thank the members there for the hospitality. 

%%%%%%%%%%%%%%%%%%%%%%%%%
%%%
\section{Overconvergent de Rham--Witt cohomology and Monsky--Washnitzer cohomology}
%%%
%%%%%%%%%%%%%%%%%%%%%%%%%

As we recalled in the introduction, the integral 
overconvergent de Rham--Witt cohomology 
$H^i(X, W^{\dagger}\Omega^{\bullet}_X)$ is defined for any smooth variety $X$ over 
$k$, but it is not so easy to compute it directly for general $X$. 
When $X$ is affine and smooth, there exists a simpler construction 
of the cohomology $H^i_{\mathrm{MW}}(X/W)$ (called the integral Monsky--Washnitzer cohomology), which is due to Monsky and Washnitzer \cite{MW_1968}. 
In this section, we briefly recall the definition of the integral Monsky--Washnitzer 
cohomology and recall the comparison theorem between the integral 
overconvergent de Rham--Witt cohomology and the 
integral Monsky--Washnitzer cohomology. 

Let $X = \Spec(\overline{A})$ be an affine smooth variety over $k$ and 
take a lift ${\mathcal{X}} = \Spec(A)$ of $X$ to an affine smooth scheme over $W$. 
(The existence of such a lift is due to Elkik \cite{E_1973}.) Let $A^{\dagger}$ be 
the weak completion of $A$ (defined by Monsky--Washnitzer), and let 
$\Omega^{\bullet}_{A^{\dagger}}$ be the de Rham complex of continuous differentials 
of $A^{\dagger}$ over $W$. We define the integral Monsky--Washnitzer cohomology
 $H^i_{\mathrm{MW}}(X/W)$ of $X$ by 
\[ H^i_{\mathrm{MW}}(X/W) := H^i(\Omega^{\bullet}_{A^{\dagger}}). \] 
It is known that this definition is independent of the choice of the lift
${\mathcal{X}} = \Spec(A)$ \cite{ES}. 
Then we have the following comparison theorem: 

\begin{theorem}[Davis--Langer--Zink \cite{DLZ_2011}, 
Davis--Zureick-Brown \cite{DZB_2014}, Ertl--Sprang \cite{ES}]\label{thm1.1}
Let $X$ be an affine smooth variety over $k$. Then there exists a canonical 
isomorphism 
$H^i_{\mathrm{MW}}(X/W) \cong H^i(X, W^{\dagger}\Omega^{\bullet}_X)$. 
\end{theorem}

As in the case of overconvergent de Rham--Witt cohomology, 
we define the integral Monsky--Washnitzer cohomology modulo torsion 
$\overline{H}^i_{\mathrm{MW}}(X/W)$ by 
\[ \overline{H}^i_{\mathrm{MW}}(X/W) := {\mathrm{Im}}
(H^i_{\mathrm{MW}}(X/W) \to H^i_{\mathrm{MW}}(X/W) \otimes \Q). \] 
Then, Theorem \ref{thm1.1} implies the isomorphism 
\begin{equation}\label{eq1}
\overline{H}^i_{\mathrm{MW}}(X/W) \cong 
\overline{H}^i(X, W^{\dagger}\Omega^{\bullet}_X). 
\end{equation}

We next recall a relation between Monsky--Washnitzer cohomology and
algebraic de Rham cohomology. Let $X = \Spec(\overline{A}), 
{\mathcal{X}} = \Spec(A)$ as before Theorem \ref{thm1.1} and let 
$\Omega^{\bullet}_A$ be the de Rham complex of algebraic differentials 
of $A$ over $W$. We define the integral algebraic de Rham cohomology 
$H^i_{\mathrm{dR}}({\mathcal{X}}/W)$ by 
\[ H^i_{\mathrm{dR}}({\mathcal{X}}/W) :=  H^i(\Omega^{\bullet}_{A}), \] 
and the integral algebraic de Rham cohomology modulo torsion by 
\[ \overline{H}^i_{\mathrm{dR}}({\mathcal{X}}/W) := {\mathrm{Im}}
(H^i_{\mathrm{dR}}({\mathcal{X}}/W) \to H^i_{\mathrm{dR}}({\mathcal{X}}/W) \otimes \Q). \]

The canonical map of weak completion $\iota: A \to A^{\dagger}$ induces a morphism
\[ \iota_*: H^i_{\mathrm{dR}}({\mathcal{X}}/W) \to H^i_{\mathrm{MW}}(X/W), \] 
hence the commutative diagram with injective vertical arrows: 
\begin{equation}\label{eq2}
\xymatrix{
\overline{H}^i_{\mathrm{dR}}({\mathcal{X}}/W) \ar@{^{(}->}[d] \ar[r]^{\overline{\iota}_*} & 
 \overline{H}^i_{\mathrm{MW}}(X/W) \ar@{^{(}->}[d] \\ 
H^i_{\mathrm{dR}}({\mathcal{X}}/W) \otimes \Q \ar[r]^{\iota_{*,\Q}} & 
H^i_{\mathrm{MW}}(X/W) \otimes \Q. 
}
\end{equation}
In general, it is not necessarily true that $\iota_{*,\Q}$ is an isomorphism. 

%%%%%%%%%%%%%%%%%%%%%%%%%
%%%
\section{$C_{ab}$-curves}
%%%
%%%%%%%%%%%%%%%%%%%%%%%%%

In this section, we give a review of the result of Denef--Vercauteren \cite[\S 3]{DV_2006}
on the computation of rational algebraic de Rham cohomology and 
rational Monsky--Washnitzer cohomology of $C_{ab}$-curves. 
(This is a generalization of the computation by Kedlaya \cite{K_2001} in the case of 
hyperelliptic curves.) 

First we recall the definition of $C_{ab}$-curves. 

\begin{definition}
Let $a, b$ be coprime positive integers and 
let $L$ be a field of characteristic prime to $ab$. 
A $C_{ab}$-curve over $L$ is an affine smooth plane curve $X$ over $L$ 
defined by an equation of the form 
\begin{equation}\label{eq3}
\overline{f}(x, y) := y^a + \sum_{j=1}^{a-1} \overline{f}_j(x) y^j + \overline{f}_0(x) = 0, 
\end{equation}
where $\overline{f}_j(x) \in L[x]$ for $0 \leq j \leq a-1$ with 
$\deg \overline{f}_0 = b$ and $a \deg \overline{f}_j + b j < ab$ for 
$1 \leq j \leq a-1$. 
\end{definition}

\begin{remark}
\begin{enumerate}
\item
In some references, the smooth compactification of $X$ is called  
a $C_{ab}$-curve. However, we adopt the above definition because 
we will not use the compactification so much. 
\item
When $\overline{f}_j(x) = 0$ for all $1 \leq j \leq a-1$, the curve $X$ is called 
a superelliptic curve (minus one point). 
When we assume moreover that $a=2$, the curve 
$X$ is called a hyperelliptic curve (minus one point of characteristic prime to 
$2b$). When we assume moreover that $b=3$, the curve 
$X$ is called an elliptic curve (minus one point of characteristic prime to 
$6$). 
\item
The smoothness assumption on $X$ is nothing but the Jacobian criterion 
associated to the equation \eqref{eq3}. In the case of superelliptic curves, 
$X$ is smooth if and only if $\overline{f}_0(x)$ does not have multiple roots in the 
algebraic closure of $L$. 
\end{enumerate}
\end{remark}

Then the following facts are known: 

\begin{fact}\label{fact1}{\rm \cite[p.81]{DV_2006}, \cite{M}}
\begin{enumerate}
\item
There exists a unique $L$-rational point at infinity (the point in the smooth compactification 
of $X$ which is not in $X$), which we denote by $P_{\infty}$. 
\item
$\ord_{P_{\infty}}(x) = -a, \,\, \ord_{P_{\infty}}(y) = -b$. 
\item
The genus of the smooth compactification of $X$ is equal to $(a-1)(b-1)/2$. 
\end{enumerate}
\end{fact}

Now let $a, b$ be coprime positive integers prime to $p$ and 
consider a $C_{ab}$-curve $X$ over $k$ defined by the 
equation \eqref{eq3} (with $L$ replaced by $k$). 

For each $0 \leq j \leq a$, we take a  lift $f_j(x) \in W[x]$ of 
$\overline{f}_j(x)$ with $\deg f_j = \deg \overline{f}_j$.  
By definition, we can write explicitly that 
\begin{equation}\label{eqcoef}
f_0(x) = \sum_{i=0}^{b}c_{i0}x^i, \quad 
f_j(x) = \sum_{ai+bj < ab}c_{ij}x^i \,\,(1 \leq j \leq a-1)
\end{equation} 
with $c_{i0}, c_{ij} \in W$. Then we have the following: 

\begin{lemma}\label{lemsm}
The equation 
\begin{equation}\label{eq4}
f(x, y) := y^a + \sum_{j=1}^{a-1} f_j(x) y^j + f_0(x) = 0 
\end{equation}
defines a smooth lift ${\mathcal{X}}$ of $X$ over $W$. 
\end{lemma} 

\begin{proof}
This is implicitly proven in \cite[\S 3]{DV_2006}, 
but we sketch the argument for the convenience of the reader. 
Let 
$F_0 = f(X^a, Y^b), F_1 = \frac{\partial f}{\partial x}(X^a, Y^b), 
F_2 = \frac{\partial f}{\partial y}(X^a, Y^b)$ 
and let $F_i^{\rm h}(X,Y,Z) \, (i=0,1,2)$ be the homogenization of $F_i$. 

Let $\overline{F}_i^{\rm h}\, (i=0,1,2)$ be the reduction modulo $p$ of $F_i^{\rm h}$. 
Then the equation $\overline{F}_0^{\rm h} = \overline{F}_1^{\rm h} 
= \overline{F}_2^{\rm h} = 0$ has no solution in the projective space 
${\mathbf{P}}^2(k^{\rm alg})$: Indeed, when $Z \not= 0$, this follows from 
the smoothness of the curve $X$ over $k$ and when $Z = 0$, the above equality becomes 
$$ Y^{ab} + X^{ab} = bX^{a(b-1)} = aY^{(a-1)b} = 0, $$
which has no nontrivial solution. 

Then, \cite[Theorem 2]{DV_2006} implies that there exist $G_0, G_1, G_2 \in W[X,Y]$ with 
$\sum_{i=0}^2 G_iF_i = 1$. We may assume that $G_0, G_1, G_2$ are 
linear combinations of the monomials $X^{ai}Y^{bj}$ $(i,j \in \N)$ because 
the equality $\sum_{i=0}^2 G_iF_i = 1$ remains true when 
we discard all the other monomials from $G_0, G_1, G_2$. 
Then, if we set 
$g_i(x,y) \in W[x,y]$ $(i=0,1,2)$ so that $g_i(X^a,Y^b) = G_i(X,Y)$, 
we see the equality 
$ g_0 f + g_1 \frac{\partial f}{\partial x} + g_2 \frac{\partial f}{\partial y} = 1$. 
So ${\mathcal{X}}$ is smooth over $W$ by the Jacobian criterion. 
\end{proof}

Let ${\mathcal{X}}_K$ be the generic fiber of ${\mathcal{X}}/W$, 
which is a $C_{ab}$-curve over $K$. Let $P_{\infty}$ be the 
point at infinity of ${\mathcal{X}}_K$. (See Fact \ref{fact1}(i).) 

We have $X = \Spec (\overline{A})$, ${\mathcal{X}} = \Spec (A)$ with 
\[ \overline{A} = k[x,y]/(\overline{f}(x,y)), \quad 
A = W[x,y]/(f(x,y)). \] 
Let $A^{\dagger}$ be the weak completion of $A$. 
By definition given in the previous section, we have the first cohomologies 
\[ 
H^1_{\mathrm{dR}}({\mathcal{X}}/W) = H^1(\Omega^{\bullet}_A), 
\quad H^1_{\mathrm{MW}}(X/W) = H^1(\Omega^{\bullet}_{A^{\dagger}}) \] 
which induce the diagram \eqref{eq2} (with $i=1$). 

Denef--Vercauteren first compute a basis of the 
first rational algebraic de Rham cohomology 
\begin{equation}\label{eq5}
H^1_{\mathrm{dR}}({\mathcal X}/W) \otimes \Q 
= H^1(\Omega^{\bullet}_{A} \otimes \Q).   
\end{equation}
To explain their result and its proof, for 
$\omega \in \Omega^{\bullet}_{A} \otimes \Q$, we denote its 
cohomology class in the groups \eqref{eq5} by $[\omega]$. 

\begin{proposition}[{\cite[p.89]{DV_2006}}]\label{DV} 
The elements $[x^iy^jdx]$  $(0 \leq i \leq b-2, 1 \leq j \leq a-1)$ form 
a basis of $H^1_{\mathrm{dR}}({\mathcal X}/W) \otimes \Q 
= H^1(\Omega^{\bullet}_{A} \otimes \Q)$ over $K$.
\end{proposition}

We give a sketch of the proof because it is important for us. 

\begin{proof}
The proof is done in several steps. 

\medskip 

\noindent
{\textbf{Step 1.}} \, 
By definition, the group $H^1(\Omega^{\bullet}_{A} \otimes \Q)$ is generated by 
$[x^iy^jdx], [x^iy^jdy]$  $(i, j \in \N)$ over $K$. 
Using the defining equation \eqref{eq4}, we see that the group $H^1(\Omega^{\bullet}_{A} \otimes \Q)$ 
is generated by $[x^iy^jdx], [x^iy^jdy]$  $(i \in \N, 0 \leq j \leq a-1)$ over $K$. 

\medskip 

\noindent
{\textbf{Step 2.}} \, 
Next, using the equality 
\begin{equation}\label{eq:revise}
0 = [d(x^iy^{j})] = j[x^iy^{j-1}dy] + i[x^{i-1}y^{j}dx]  
\end{equation}
and the defining equation \eqref{eq4} again if necessary, we see that 
the group $H^1(\Omega^{\bullet}_{A} \otimes \Q)$ 
is generated by $[x^iy^jdx]$  $(i \in \N, 1 \leq j \leq a-1)$ over $K$. 

\medskip 

\noindent
{\textbf{Step 3.}} \, For each $j, l \in \N$, 
Denef--Vercauteren prove the following equality \cite[(18) in p.89]{DV_2006}: 
\begin{equation}\label{eq6}
\left[ 
x^l \left( \sum_{k=1}^{a-1} \frac{j}{k+j} f'_k(x) y^k + f'_0(x) \right) y^j dx 
- 
lx^{l-1}\left( \frac{a}{a+j} y^a + \sum_{k=1}^{a-1} \frac{k}{k+j} f_k(x) y^k 
\right) y^j dx 
\right] = 0. 
\end{equation}
We compute the order at $P_{\infty}$ of terms in the equality \eqref{eq6}, 
noting that $\ord_{P_{\infty}}(x) = -a, \ord_{P_{\infty}}(y) = -b, 
\ord_{P_{\infty}}(dx) = -(a+1)$. (See Fact \ref{fact1}(ii).) 
The possible terms with lowest order are 
$\omega_1 = x^lf'_0(x)y^jdx$ and $\omega_2 = - lx^{l-1} \frac{a}{a+j} y^{a+j}dx$ 
and the order is $-(a(l+b) + jb + 1)$. 
Because 
\begin{align*}
& \omega_1 = b c_{b0} x^{l+b-1}y^j dx + \text{(forms of higher order)}, \\  
& \omega_2 = \frac{la}{a+j} c_{b0} x^{l+b-1}y^jdx + \text{(forms of higher order)}
\end{align*}
(we used \eqref{eq4} for $\omega_2$), we conclude that the differential form 
inside the bracket $[ \phantom{a} ]$ in \eqref{eq6} has the form 
\[ \left( b + \frac{la}{a+j} \right) c_{b0} x^{l+b-1}y^j dx + \text{(forms of higher order)}. \]
Because the coefficient $(b + \frac{la}{a+j})c_{b0}$ is nonzero, 
the order of the differential form 
inside the bracket $[ \phantom{a} ]$ in \eqref{eq6} is equal to $-(a(l+b) + jb + 1)$. 

Now we go back to consider $[x^iy^jdx]$  $(i \in \N, 1 \leq j \leq a-1)$. 
The order at $P_{\infty}$ of $x^iy^jdx$ is $-(a(i+1)+jb+1)$. So, if $i \geq b-1$, 
we can use \eqref{eq6}, \eqref{eq4} and \eqref{eq:revise} to rewrite $[x^iy^jdx]$ 
as a linear combination over $K$ of the elements $[x^{i'}y^{j'}dx]$  $(i' \in \N, 1 \leq j' \leq a-1)$
such that the order at $P_{\infty}$ of $x^{i'}y^{j'}dx$ is strictly larger than 
that of $x^iy^jdx$. If $i' \geq b-1$ for some $[x^{i'}y^{j'}dx]$ appearing in the linear 
combination, we use again \eqref{eq6}, \eqref{eq4} and \eqref{eq:revise} to rewrite this term. 
We repeat this process as long as there is a term 
$[x^{i'}y^{j'}dx]$ with $i' \geq b-1$ in the linear combination. 
Because the order at $P_{\infty}$ is always an integer, this process stops at some point 
and so we conclude that, for $i \geq b-1$ and $1 \leq j \leq a-1$, 
$[x^iy^jdx]$ is written as a linear combination of 
the elements $[x^{i'}y^{j'}dx]$ with $0 \leq i' \leq b-2, 1 \leq j' \leq a-1$. 
Hence 
the group $H^1(\Omega^{\bullet}_{A} \otimes \Q)$ 
is generated by $[x^iy^jdx]$  $(0 \leq i \leq b-2, 1 \leq j \leq a-1)$ over $K$. 
Because the genus of the smooth compactification of $X$ is $(a-1)(b-1)/2$ (see Fact \ref{fact1}(iii)), 
we see that the elements $[x^iy^jdx]$  $(0 \leq i \leq b-2, 1 \leq j \leq a-1)$
form a basis of $H^1(\Omega^{\bullet}_{A} \otimes \Q)$. 
\end{proof}

\begin{remark}\label{remDV}
Let $i,j \in \N$. 
By Proposition \ref{DV}, $[x^iy^jdx]$ is written uniquely in the form 
\begin{equation}\label{eq8}
[x^iy^jdx] = \sum_{0 \leq i' \leq b-2 \atop 1 \leq j' \leq a-1} g^{i,j}_{i',j'} [x^{i'}y^{j'}dx] 
\end{equation} 
with $g^{i,j}_{i',j'} \in K$. By looking at the proof of Proposition \ref{DV}, we see that 
each $g^{i,j}_{i',j'}$ is a polynomial function in the coefficients $c_{st}  (s, t \geq 0, as + bt < ab \text{ or } (s,t) = (b,0))$ 
(see \eqref{eqcoef}) appearing in the defining equation \eqref{eq4} of ${\mathcal{X}}$, 
divided by some power of $c_{b0}$: 
Namely, there exist $G^{i,j}_{i',j'} \in \Q[z_{st}]_{s, t \geq 0, as + bt < ab \text{ or } (s,t) = (b,0)}[z_{b0}^{-1}]$ 
for any $i,j,i',j' \in \N, 0 \leq i' \leq b-2, 1 \leq j' \leq a-1$ 
such that $g^{i,j}_{i',j'}$ is the value of $G^{i,j}_{i',j'}$ at $z_{st} = c_{st}$. 
\end{remark}

Next Denef--Vercauteren compute a basis of the 
first rational Monsky--Washnitzer cohomology 
\begin{equation}\label{eq9}
H^1_{\mathrm{MW}}(X/W) \otimes \Q 
= H^1(\Omega^{\bullet}_{A^{\dagger}} \otimes \Q).  
\end{equation}
Following the previous notation, for 
$\omega \in \Omega^{\bullet}_{A^{\dagger}} \otimes \Q$, we denote its 
cohomology class in the groups \eqref{eq9} by $[\omega]$. 

\begin{proposition}[{\cite[p.90--93]{DV_2006}}]\label{DV2} 
The elements $[x^iy^jdx]$  $(0 \leq i \leq b-2, 1 \leq j \leq a-1)$ form 
a basis of 
$H^1_{\mathrm{MW}}(X/W) \otimes \Q 
= H^1(\Omega^{\bullet}_{A^{\dagger}} \otimes \Q)$ 
over $K$.
\end{proposition}

We omit the proof of Proposition \ref{DV2} because it is not necessary for us. 

\begin{corollary}
For a $C_{ab}$-curve $X$, the map $\iota_{*,\Q}$ in the diagram \eqref{eq2} 
is an isomorphism. 
\end{corollary}

\begin{proof}
Because the map $\iota_{*,\Q}$ sends $[x^iy^jdx]$ to $[x^iy^jdx]$, 
the claim follows from Propositions \ref{DV} and \ref{DV2}. 
\end{proof}

%%%%%%%%%%%%%%%%%%%%%%%%%
%%%
\section{Infiniteness}
%%%
%%%%%%%%%%%%%%%%%%%%%%%%%

In this section, we give examples of affine smooth varieties $X$ over $k$ 
such that the first integral overconvergent de Rham--Witt cohomology modulo torsion 
$\overline{H}^1(X, W^{\dagger}\Omega^{\bullet}_X)$ is not finitely generated over $W$. 
Our basic example is the following one: 

\begin{theorem}\label{main1}
Let $a, b \geq 2$ be coprime integers prime to $p$, let $\overline{\alpha} \in k^{\times}$ 
and let $X$ be the superelliptic curve $y^a + x^b + \overline{\alpha} = 0$ 
(the affine smooth plane curve defined by this equation). 
Then $\overline{H}^1(X, W^{\dagger}\Omega^{\bullet}_X)$ is not finitely generated over $W$. 
\end{theorem}

\begin{proof}
Note that $X$ is a special case (the case $\overline{f}_0(x) = x^b + \overline{\alpha}, \overline{f}_j(x) = 0$ 
$(1 \leq j \leq a-1)$) of the $C_{ab}$-curve $X$ in the previous section. 
Take a lift $\alpha \in W$ of $\overline{\alpha}$ and let ${\mathcal{X}}$ be 
the smooth lift of $X$ defined by the equation $y^a + x^b + \alpha = 0$. 
This is a special case (the case $f_0(x) = x^b + \alpha, \overline{f}_j(x) = 0$ 
$(1 \leq j \leq a-1)$) of the lift ${\mathcal{X}}$ in the previous section. Let 
$A = W[x,y]/(y^a + x^b + \alpha)$ be as in the previous section. 

We consider the algebraic de Rham cohomology $H^1_{\mathrm{dR}}({\mathcal{X}}/W) \otimes \Q = 
H^1(\Omega^{\bullet}_{A} \otimes \Q)$. The equality \eqref{eq6} in the cohomology 
is written as 
\[ \left[ 
b x^{l+b-1} y^j dx 
- \frac{la}{a+j} x^{l-1} y^{a+j} dx \right] = 0
\] 
in the case at hand. Using the defining equation $y^a + x^b + \alpha = 0$, 
it is rewritten as 
\[ 
\left[ b x^{l+b-1} y^j dx 
+ \frac{la}{a+j} x^{l-1}(x^b + \alpha) y^{j} dx \right] = 0, 
\] 
which is equivalent to the equality 
\begin{equation}\label{eq10}
\left[ x^{l+b-1}y^jdx\right] = 
\left[ - \frac{la\alpha}{la + jb + ab} x^{l-1}y^j dx \right]. 
\end{equation}

Recall that $1 \leq j \leq a-1$. Let $1 \leq r \leq b-1$, $N \in \N$ and 
consider the element $[x^{(N+1)b+(r-1)}y^jdx]$. 
By using the equality \eqref{eq10} with 
$l = Nb+r, (N-1)b+r, \dots, r$, we obtain the equality 
\begin{align}
[x^{(N+1)b+(r-1)}y^jdx] & = 
\prod_{n=0}^{N} \left( - \frac{(nb+r)a\alpha}{(nb+r)a + jb + ab} \right) 
[x^{r-1}y^j dx] \label{eq11} \\ 
& = \text{(unit)} \cdot 
\prod_{n=0}^{N} \frac{nb+r}{nab+(ra + jb + ab)}  \cdot 
[x^{r-1}y^j dx], \nonumber 
\end{align}
where $\text{(unit)}$ means an element in $W^{\times}$. 

Now we fix $j, r$ so that $p \equiv jb \,({\mathrm{mod}} \, a)$, 
$p \equiv ra \,({\mathrm{mod}} \, b)$. This is possible because 
$a, b$ are coprime and $p$ does not divide $ab$. Then we have 
$p \equiv ra + jb + ab \,({\mathrm{mod}} \, ab)$. Hence the set 
\[ {\mathbf{M}} := \{M \in \N\,|\, p^M \geq ra + jb + ab, \,\, p^M \equiv ra + jb + ab \,({\mathrm{mod}} \, ab) \} \] 
is infinite. Take any $M \in  {\mathbf{M}}$ and put 
$N := (p^M - (ra + jb + ab))/ab$. For such $N$, we compute the $p$-adic valuation 
$\nu := \nu_p \left( \displaystyle\prod_{n=0}^{N} \frac{nb+r}{nab+(ra + jb + ab)} \right)$
of $ \displaystyle\prod_{n=0}^{N} \frac{nb+r}{nab+(ra + jb + ab)}$. 

For $M' \in \N$, define the sets $P_{M'}, Q_{M'}$ by 
\[ P_{M'} := \{n \,|\, 0 \leq n \leq N, \,\, p^{M'} | nb + r \}, \quad 
Q_{M'} := \{n \,|\, 0 \leq n \leq N, \,\, p^{M'} | nab + (ra + jb + ab) \}. 
 \] 
Then 
\begin{align}
\nu & = \sum_{M'=1}^{\infty} M'( (|P_{M'}| - |P_{M'+1}|) - (|Q_{M'}| - |Q_{M'+1}|) ) \label{eq12} \\ 
& = \sum_{M'=1}^{\infty} (|P_{M'}| - |Q_{M'}|) = \sum_{M'=1}^{M} (|P_{M'}| - |Q_{M'}|). \nonumber
\end{align}
(For the last equality, note that $Nb+r \leq Nab + (ra+jb+ab) = p^M$.) 
So we estimate the terms $|P_{M'}| - |Q_{M'}|$ for $1 \leq M' \leq M$. 

In general, for $n_0 \in P_{M'}$ and $0 \leq n \leq N$, we have the equivalence 
\[ n \in P_{M'} \iff n-n_0 \in p^{M'}{\mathbf{Z}} \] 
because $b$ is prime to $p$, and the same property holds for the set $Q_{M'}$. 
So, if we denote the maximal element of $P_{M'}$ by $n_0$, we have the equality 
\[ P_{M'} = \{n_0, n_0 - p^{M'}, n_0 - 2 p^{M'}, \dots \} \cap \N. \]
Also, since $N \in Q_{M'}$ (because $p^{M'} | p^M = Nab + (ra + jb + ab)$), we have the equality 
\[ Q_{M'} = \{N, N - p^{M'}, N - 2 p^{M'}, \dots \} \cap \N. \]
Since $n_0 \leq N$ by definition, we see that 
$|P_{M'}| \leq |Q_{M'}|$. So $|P_{M'}| - |Q_{M'}| \leq 0$ for any $1 \leq M' \leq M$.  

Next we give a stronger estimate of $|P_{M'}| - |Q_{M'}|$ for 
$1 \leq M' \leq M$ when $M' \in {\mathbf{M}}$. 
If we define the sets $\widetilde{P}_{M'}$, $\widetilde{Q}_{M'}$ by 
\[ \widetilde{P}_{M'} := \{ a(nb + r) \,|\, n \in P_{M'}\}, \quad 
\widetilde{Q}_{M'} := \{ nab + (ra + jb + ab) \,|\, n \in Q_{M'}\}, \]
$|P_{M'}| - |Q_{M'}| = |\widetilde{P}_{M'}| - |\widetilde{Q}_{M'}|$. 
If we denote the maximal (resp. minimal) element of $P_{M'}$ by 
$n_0$ (resp. $n_1$) and put $\widetilde{n}_0 := a(n_0b + r)$ 
(resp. $\widetilde{n}_1 := a(n_1b + r)$), we have the equality 
\[ \widetilde{P}_{M'} = \{\widetilde{n}_0, \widetilde{n}_0 - ab p^{M'}, \widetilde{n}_0 - 2 ab p^{M'}, \dots, \widetilde{n}_1\}. \]
On the other hand, $N$ is the maximal element of $Q_{M'}$ and since 
$M' \in {\mathbf{M}}$, there exists $0 \leq N_1 \leq N$ with 
$N_1ab + (ra + jb + ab) = p^{M'}$; hence $N_1$ is the minimal element of $Q_{M'}$. 
Then 
\[ Q_{M'} = \{N, N - p^{M'}, N - 2 p^{M'}, \dots, N_1 \}, \]
and so we see the equality 
\[ \widetilde{Q}_{M'} = \{p^M, p^M - abp^{M'}, p^M - 2abp^{M'}, \dots, p^{M'} \}. \] 
Now, noting the inequalities 
\[ \widetilde{n}_1 = a(n_1b + r) > n_1b + r \geq p^{M'}, 
\quad \widetilde{n}_0 = a(n_0b + r) \leq a(Nb+r) < Nab + (ra + jb + ab) = p^{M}, \]
we see that $|\widetilde{P}_{M'}| < |\widetilde{Q}_{M'}|$. 
So $|P_{M'}| - |Q_{M'}| \leq -1$ for any $1 \leq M' \leq M$ with $M' \in {\mathbf{M}}$. 

By combining the inequalites proved in the previous two paragraphs and the equality \eqref{eq12}, 
we see that, if $M$ is the $d$-th element of the set ${\mathbf{M}}$, 
$\nu \leq -d$. Then, by putting it into the equality \eqref{eq11}, we see that, 
for some fixed $1 \leq j \leq a-1$ and $1 \leq r \leq b-1$, there exists a 
sequence of natural numbers $\{N_d\}_{d=0}^{\infty}$ such that 
\begin{equation}\label{eq13}
[x^{(N_d+1)b+(r-1)}y^jdx] = C_d [x^{r-1}y^j dx], \quad \nu_p(C_d) \leq -d.  
\end{equation}
Because $[x^{(N_d+1)b+(r-1)}y^jdx]$ is the cohomology class coming from 
the integral algebraic de Rham cohomology $H^1_{\mathrm{dR}}({\mathcal{X}}/W)$, 
we see from 
\eqref{eq13} that 
$\overline{H}^1_{\mathrm{dR}}({\mathcal{X}}/W)$ contains $K [x^{r-1}y^j dx] \cong K$. 
(The last isomorphism follows from Proposition \ref{DV}.) 
Then, by the diagram \eqref{eq2} and the fact that $\iota_{*,\Q}$ is an isomorphism, 
we conclude that $\overline{H}^1_{\mathrm{MW}}(X/W)$ also 
contains $K [x^{r-1}y^j dx] \cong K$. Since the last cohomology group is isomorphic to 
$\overline{H}^1(X, W^{\dagger}\Omega^{\bullet}_X)$, we conclude that 
$\overline{H}^1(X, W^{\dagger}\Omega^{\bullet}_X)$ is not finitely generated over $W$. 
\end{proof}

\begin{corollary}\label{cor}
For any prime number $p$ and any perfect field $k$ of characteristic $p$, 
there exist infinitely many affine smooth curves $X_i$ $(i \in \N)$ over $k$ 
whose smooth compactifications are all non-isomorphic 
such that 
$\overline{H}^1(X_i, W^{\dagger}\Omega_{X_i}^{\bullet}) \, (i \in \N)$ are not finitely generated 
over $W$. 
%For any prime number $p$ and any perfect field $k$ of characteristic $p$, 
%there exists an affine smooth curve $X$ over $k$ such that 
%$\overline{H}^1(X, W^{\dagger}\Omega_X^{\bullet})$ is not finitely generated 
%over $W$. 
\end{corollary}

\begin{proof}
If we take a sequence of coprime integers $a_i, b_i \geq 2 \, (i \in \N)$ 
prime to $p$ with $(a_i-1)(b_i-1)/2 < (a_{i+1}-1)(b_{i+1}-1)/2$ and take as $X_i$ 
the superelliptic curve 
$y^{a_i} + x^{b_i} + 1 = 0$, 
$\overline{H}^1(X_i, W^{\dagger}\Omega^{\bullet}_{X_i}) \, (i \in \N)$ are 
not finitely generated over $W$ and smooth compactifications of $X_i$'s are 
all non-isomorphic because they have different genera. 
\end{proof}

It would be possible to provide 
more examples of infiniteness using the following:  

\begin{proposition}\label{main3}
If $X_1 \rightarrow X_2$ be a generically \'etale morphism of affine smooth curves. 
Then, if $\overline{H}^1(X_2, W^{\dagger}\Omega^{\bullet}_X)$ is not finitely generated over $W$, 
$\overline{H}^1(X_1, W^{\dagger}\Omega^{\bullet}_X)$ is not finitely generated over $W$ either. 
\end{proposition}
\begin{proof}
One can take open subschemes $X'_1 \subset X_1, X'_2 \subset X_2$ such that 
$X'_1 \to X'_2$ is finite \'etale. Since we have the commutative diagram 
$$
\xymatrix{
\overline{H}^1(X_2, W^{\dagger}\Omega^{\bullet}_X) 
\ar@{^{(}->}[r] \ar[d] & H^1_{\rig}(X_2/K) \ar@{^{(}->}[r] \ar[d] & 
H^1_{\rig}(X'_2/K) \ar@{^{(}->}[d]
\\
\overline{H}^1(X_1, W^{\dagger}\Omega^{\bullet}_X) \ar@{^{(}->}[r]  
& H^1_{\rig}(X_1/K) \ar@{^{(}->}[r] & 
H^1_{\rig}(X'_1/K)
} 
$$
with horizontal arrows and the right vertical arrow injective, 
we see that the morphism $\overline{H}^1(X_2, W^{\dagger}\Omega^{\bullet}_X)  
\rightarrow \overline{H}^1(X_1, W^{\dagger}\Omega^{\bullet}_X) $ is injective as well, and hence the claim follows.
\end{proof}

%The following corollary is a stronger version of Corollary \ref{cor}. 
%
%\begin{corollary}\label{cor2}
%For any prime number $p$ and any perfect field $k$ of characteristic $p$, 
%there exist infinitely many affine smooth curves $X_i$ $(i \in \N)$ over $k$ 
%whose compactifications are all non-isomorphic 
%such that 
%$\overline{H}^1(X_i, W^{\dagger}\Omega_{X_i}^{\bullet}) \, (i \in \N)$ are not finitely generated 
%over $W$. 
%\end{corollary}

We give another proof of Corollary \ref{cor}, using Proposition \ref{main3}. 

\begin{proof}[{\sc Another proof of Corollary \ref{cor}}]
We take as $X$ the superelliptic curve $y^{a} + x^{b} + 1 = 0$ such that 
the genus $g(X^{\rm cpt}) = (a-1)(b-1)/2$ of the smooth compactification $X^{\rm cpt}$ of 
$X$ is $\geq 2$. Then, by taking a family 
$f_i: X_i^{\rm cpt} \to X^{\rm cpt} \,(i \in \N)$ of finite \'etale coverings  
with $g(X_{i+1}^{\rm cpt}) > g(X_i^{\rm cpt})$ and putting $X_i := f_i^{-1}(X)$, 
we obtain the required family $X_i$ $(i \in \N)$ thanks to Proposition \ref{main3}. 
\end{proof}

Next we consider the case of `general' $C_{ab}$-curves. 
Let $p$ be a fixed prime, let $a, b$ be coprime positive integers prime to $p$ and 
let $k_0$ be a field of characteristic $p$. 
Let $S$ be the set of pairs $(i,j) \in \N^2$ with $ai + bj < ab$ or $(i,j) = (b,0)$, and 
let $\Spec (k_0[\overline{z}_{ij}]_{(i,j) \in S}) = {\mathbf{A}}^{|S|}_{k_0}$ be the 
affine space with respect to the indeterminates $\{\overline{z}_{ij}\}_{(i,j) \in S}$. 
Let $g: {\mathfrak{X}} \to {\mathbf{A}}_{k_0}^{|S|}$ be the relative plane curve defined by the 
equation \eqref{eq3} with 
\begin{equation*}
\overline{f}_0(x) = \sum_{i=0}^{b}\overline{z}_{i0}x^i, \quad 
\overline{f}_j(x) = \sum_{ai+bj < ab}\overline{z}_{ij}x^i \,\, (1 \leq j \leq a-1).  
\end{equation*}
For a point ${\mathbf{c}}$ in ${\mathbf{A}}_{k_0}^{|S|}$, let  
${\mathfrak{X}}_{\mathbf{c}}$ be the fiber of the map 
$g$ at ${\mathbf{c}}$. Then we have the following proposition:

\begin{proposition}\label{gensm}
There exists an open dense subscheme $U$ of ${\mathbf{A}}^{|S|}_{k_0}$
such that a point ${\mathbf{c}} \in {\mathbf{A}}^{|S|}_{k_0}$ belongs to $U$ 
if and only if   
${\mathfrak{X}}_{\mathbf{c}}$ is a (smooth) $C_{ab}$-curve.  
\end{proposition}

\begin{proof}
The following proof is inspired by the proof of \cite[Proposition 1]{CDV_2006}. 
Let $Z$ be the closed subscheme of ${\mathbf{A}}_{k_0}^{|S|} \times ({\mathbf{A}}_{k_0} \setminus \{0\})^2$ 
(with the coordinates $\overline{z}_{ij} \, ((i,j) \in S), x, y$) 
defined by the equation $\overline{f} = \frac{\partial \overline{f}}{\partial x} 
= \frac{\partial \overline{f}}{\partial y} = 0$, regarded as the equation in 
$\vert S\vert + 2$ variables $\overline{z}_{ij} \, ((i,j) \in S), x, y$. 

First we prove that, for each $\alpha, \beta \in k_0^{\rm alg} \setminus \{0\}$, the pullback 
$Z_{\alpha, \beta}$ of $Z$ to 
${\mathbf{A}}_{k_0^{\rm alg}}^{|S|} \times \{(\alpha, \beta)\}$ is a closed subscheme of  
codimension $\geq 3$ in ${\mathbf{A}}_{k_0^{\rm alg}}^{|S|} \times \{(\alpha, \beta)\}$: 
Note that we have 
$$\overline{f} = y^a + \sum_{(i,j) \in S} \overline{z}_{ij}x^iy^j, \quad  
\frac{\partial \overline{f}}{\partial x} = \sum_{(i,j) \in S} \overline{z}_{ij}i x^{i-1}y^j, 
\quad 
\frac{\partial \overline{f}}{\partial y} = ay^{a-1} + \sum_{(i,j) \in S} \overline{z}_{ij}j x^iy^{j-1}. $$
Hence, if we set $\widetilde{S} := S \cup \{(0,a)\}$, $Z_{\alpha, \beta}$ is isomorphic to 
the closed subscheme 
\begin{equation}\label{eq14}
\left\{(\overline{z}_{i,j})_{i,j \in \widetilde{S}} \,\left|\, 
\sum_{(i,j) \in \widetilde{S}} \overline{z}_{ij} \alpha^i\beta^j = 
\sum_{(i,j) \in \widetilde{S}} \overline{z}_{ij} i \alpha^{i-1}\beta^j = 
\sum_{(i,j) \in \widetilde{S}} \overline{z}_{ij} j \alpha^i\beta^{j-1} = 0 \right. \right\} \cap 
\{(\overline{z}_{ij})_{i,j \in \widetilde{S}} \,|\, \overline{z}_{0a} = 1\} 
\end{equation}
in ${\mathbf{A}}_{k_0^{\rm alg}}^{|\widetilde{S}|} \cap \{(\overline{z}_{ij})_{i,j \in \widetilde{S}} \,|\, \overline{z}_{0a} = 1\}$ 
via the natural isomorphism 
\[ {\mathbf{A}}_{k_0^{\rm alg}}^{|\widetilde{S}|} \cap \{(\overline{z}_{ij})_{i,j \in \widetilde{S}} \,|\, \overline{z}_{0a} = 1\} 
\cong {\mathbf{A}}_{k_0^{\rm alg}}^{|S|} \cong {\mathbf{A}}_{k_0^{\rm alg}}^{|\widetilde{S}|} \times \{(\alpha, \beta)\}. \] 
Since the vectors $(\alpha^i\beta^j, i\alpha^{i-1}\beta^j, j\alpha^i\beta^{j-1})$ for 
$(i,j) = (b,0), (0,a), (0,0)$ are linearly independent, we see that the former set of 
\eqref{eq14} is a linear subscheme of codimension $3$ in 
${\mathbf{A}}_{k_0^{\rm alg}}^{|\widetilde{S}|}$, and by taking intersection with the latter set 
in \eqref{eq14}, we see that $Z_{\alpha,\beta}$ is a closed subscheme of codimension 
$\geq 3$ in ${\mathbf{A}}_{k_0^{\rm alg}}^{|S|} \times \{(\alpha, \beta)\}$, as required. 

Since $Z_{\alpha,\beta}$ is a closed subscheme of codimension 
$\geq 3$ in ${\mathbf{A}}_{k_0^{\rm alg}}^{|S|} \times \{(\alpha, \beta)\}$ for any $(\alpha,\beta)$, 
$Z$ is a closed subscheme of codimension $\geq 3$ in 
${\mathbf{A}}_{k_0}^{|S|} \times ({\mathbf{A}}_{k_0} \setminus \{0\})^2$. 
If we define $Y_1$ to be the image of $Z$ by the projection 
${\mathbf{A}}_{k_0}^{|S|} \times ({\mathbf{A}}_{k_0} \setminus \{0\})^2 \to 
{\mathbf{A}}_{k_0}^{|S|}$, it is a closed subscheme of codimension $\geq 1$, 
and this is the set of points ${\mathbf{c}}$ in ${\mathbf{A}}_{k_0}^{|S|}$ such that 
the fiber ${\mathfrak{X}}_{\mathbf{c}}$ 
is not smooth at some point $(x,y)$ with $x \not= 0, y \not= 0$. 

On the other hand, let $Y'_2$ be the closed subscheme of codimension $1$ in 
${\mathbf{A}}_{k_0}^{|S|}$ 
defined by $\{z_{b0} = 0\}$, and let $Y'_3$ be the closed subscheme of codimension $1$ in 
${\mathbf{A}}_{k_0}^{|S|} \setminus Y'_2$ by  
$\{{\rm disc}(\overline{f}(x,0)) = 0 \} \cup 
\{{\rm disc}(\overline{f}(0,y)) = 0\}$. 
$Y'_2$ is the set of points ${\mathbf{c}}$ in ${\mathbf{A}}_{k_0}^{|S|} \setminus Y'_2$ 
such that the defining equation of ${\mathfrak{X}}_{\mathbf{c}}$ has degree $< b$ in $x$, and 
$Y'_3$ is the set of points ${\mathbf{c}}$ in ${\mathbf{A}}_{k_0}^{|S|} \setminus Y'_2$ 
such that ${\mathfrak{X}}_{\mathbf{c}}$ is not smooth at some point $(x,y)$ with $x = 0$ or $y = 0$. 
If we define $Y_2$ to be the union of $Y'_2$ and the closure of $Y'_3$ in ${\mathbf{A}}_{k_0}^{|S|}$, 
it is of codimension $1$ in ${\mathbf{A}}_{k_0}^{|S|}$. 

Now let $U := {\mathbf{A}}_{k_0}^{|S|} \setminus (Y_1 \cup Y_2)$. Then $U$ is dense open in 
${\mathbf{A}}_{k_0}^{|S|}$ and a point ${\mathbf{c}} \in {\mathbf{A}}_{k_0}^{|S|}$ belongs to $U$ 
if and only if  
${\mathfrak{X}}_{\mathbf{c}}$ is a (smooth) $C_{ab}$-curve. 
So the proof of the proposition is finished. 
\end{proof}

We denote the pullback of the map $g: {\mathfrak{X}} \to {\mathbf{A}}_{k_0}^{|S|}$ to $U$ 
by $g_U: {\mathfrak{X}}_U \to U$. This is a family of $C_{ab}$-curves. 
Then we can prove the following infiniteness result using Theorem \ref{main1}. 

\begin{theorem}\label{main2}
Let the notations be as above. Then, there exists a sequence of 
closed subschemes $Z_d \, (d \in \N)$ of codimension $\geq 1$ in $U$ 
satisfying the following: For any perfect field $k$ containing $k_0$ and 
for any morphism 
$h: \Spec(k) \to U$ whose image is not contained in $\bigcup_{d \in \N} (\bigcap_{d' =d}^{\infty} Z_{d'})$, 
if we denote the pullback of $g_U$ with respect to $h$ by $X \to \Spec(k)$, then 
$\overline{H}^1(X, W^{\dagger}\Omega^{\bullet}_X)$ is not finitely generated over $W$. 
\end{theorem}

Theorem \ref{main2} is applicable when $h$ factors as 
$\Spec(k) \to \Spec(k_0(U)) \to U$, 
where the first morphism is 
induced by an inclusion $k_0(U) \hookrightarrow k$ of the function field $k_0(U)$ of $U$ to 
a perfect field $k$ and the second morphism is the generic point of $U$. 
% the image of $h$ is the generic point of $U$. 
Also, when $k_0$ is uncountable, the set $U \setminus \bigcup_{d \in \N} (\bigcap_{d' =d}^{\infty} Z_{d'})$ 
contains uncountably many closed points. In these senses, 
$\overline{H}^1(X,W^{\dagger}\Omega^{\bullet}_X)$ is not finitely generated over $W$
for  a `general' $C_{ab}$-curve $X$. 

\begin{proof}
As before, let $S$ be the set of pairs $(i,j) \in \N^2$ with $ai + bj < ab$ or $(i,j) = (b,0)$. 
Take a perfect field $k$ containing $k_0$ and a morphism 
$h: \Spec(k) \to U (\hookrightarrow {\mathbf{A}}^{|S|}_{k_0} = \Spec(k_0[\overline{z}_{ij}]_{(i,j) \in S}))$. 
Let $\overline{c}_{ij} \,((i,j) \in S)$ be the image of 
$\overline{z}_{ij}$ by $h^*: k_0[\overline{z}_{ij}]_{(i,j) \in S} \to k$ and let 
$c_{ij} \in W = W(k)$ be a lift of $\overline{c}_{ij}$. 

Consider the relative curve ${\mathcal{X}}$ over $W$ defined by 
\eqref{eqcoef}, \eqref{eq4}. By 
Lemma \ref{lemsm}, ${\mathcal{X}}$ is smooth over $W$ and the generic fiber 
${\mathcal{X}}_K$ of ${\mathcal{X}}/W$ is a $C_{ab}$-curve. 
Take $1 \leq r \leq b-1, 1 \leq j \leq a-1$ and 
the sequence of natural numbers $\{N_d\}_{d=0}^{\infty}$ as in the proof of Theorem \ref{main1}. 
Let $i_d := (N_d+1)b+(r-1)$. Then, by Remark \ref{remDV}, we can write 
\begin{equation}\label{eq15}
[x^{i_d}y^jdx] = \sum_{0 \leq i' \leq b-2 \atop 1 \leq j' \leq a-1} G^{i_d,j}_{i',j'}(c_{st}) [x^{i'}y^{j'}dx] 
\end{equation} 
for some $G^{i_d,j}_{i',j'} \in \Q[z_{st}]_{(i,j) \in S}[z_{b0}^{-1}]$. 
(Note that $G^{i_d,j}_{i',j'}$ is independent of the choice of $h$ and the  choice of 
  lifts $\{c_{st}\}_{(s,t) \in S}$.)  
By the proof of Theorem \ref{main1}, the equality \eqref{eq13} is the one we obtain 
from \eqref{eq15} by specializing $z_{st} \, ((s,t) \in S)$ as 
\begin{equation}\label{eq16}
z_{st} \mapsto 0 \,\, ((s,t) \not= (0,0), (b,0)), \quad 
z_{00} \mapsto \alpha \in W^{\times}, \quad  z_{b0} \mapsto 1. 
\end{equation}

Define $l_d$ to be the least integer such that $p^{l_d}G^{i_d,j}_{r,j} \in
 {\mathbf{Z}}_p[z_{st}]_{(i,j) \in S}[z_{b0}^{-1}]$. 
Then, by the specialization \eqref{eq16}, 
$p^{l_d}G^{i_d,j}_{r,j}$ is sent to $p^{l_d}C_d$, where $C_d$ is as in \eqref{eq13}. 
Because this specialization takes value in ${\mathbf{Z}}_p$, 
we see that $p^{l_d}C_d \in {\mathbf{Z}}_p$, hence 
$l_d - d \geq l_d + \nu_p (C_d) \geq 0.$ Thus $l_d \geq d$. 

Let $\overline{G}^{i_d,j}_{r,j} \not= 0$ be the image of 
$p^{l_d}G^{i_d,j}_{r,j} \in {\mathbf{Z}}_p[z_{st}]_{(i,j) \in S}[z_{b0}^{-1}]$ by the reduction map  
$ {\mathbf{Z}}_p[z_{st}]_{(i,j) \in S}[z_{b0}^{-1}] \to 
{\mathbf{F}}_p[\overline{z}_{st}]_{(i,j) \in S}[\overline{z}_{b0}^{-1}]$, and let 
$Z_d$ be the zero locus of $\overline{G}^{i_d,j}_{r,j}$ in $U$. 
Note that $\overline{z}_{b0}$ is invertible in $U$ and so $Z_d$ is well-defined as a closed subscheme of $U$, and it is of codimension $\geq 1$ because 
 $\overline{G}^{i_d,j}_{r,j}$ is nonzero. Also, if the image of $h$ is not contained in $Z_d$, 
then $\overline{G}^{i_d,j}_{r,j}(\overline{c}_{st})$ is nonzero in $k$ and so 
$\nu_p(G^{i_d,j}_{i,j}(c_{st})) = -l_d \leq -d$. 

Now suppose that the image of $h$ is not contained in $\bigcup_{d \in \N} (\bigcap_{d' =d}^{\infty} Z_{d'})$. 
Because $[x^{i_d}y^jdx]$ is a cohomology class coming from 
the integral algebraic de Rham cohomology $H^1({\mathcal{X}}/W)$, we see from 
\eqref{eq15} and the calculation in the previous paragraph that 
$\overline{H}^1_{\rm{dR}}({\mathcal{X}}/W)$ is not contained in any of 
$p^{-d}(\bigoplus_{0 \leq i \leq b-2 \atop 1 \leq j \leq a-1}[x^iy^j dx])$ $(d \in \N)$
and so $\overline{H}^1_{\mathrm{MW}}(X/W)$ is not contained in any of 
 $p^{-d}(\bigoplus_{0 \leq i \leq b-2 \atop 1 \leq j \leq a-1}[x^iy^j dx])$ $(d \in \N)$. 
Since the last cohomology group is isomorphic to 
$\overline{H}^1(X, W^{\dagger}\Omega^{\bullet}_X)$, we conclude that 
$\overline{H}^1(X, W^{\dagger}\Omega^{\bullet}_X)$ is not finitely generated over $W$. 
\end{proof}

\begin{remark}
If infinitely many  
$Z_d$'s ($d \in {\mathbf{N}}$) intersect properly, the set 
$\bigcup_{d \in \N} (\bigcap_{d' =d}^{\infty} Z_{d'})$ 
in Theorem \ref{main2} is empty. 
Thus we expect that 
$\overline{H}^1(X, W^{\dagger}\Omega^{\bullet}_X)$ would not be 
finitely generated over $W$ for any $C_{ab}$-curve $X$.  
\end{remark}

Our results suggest that, for 
most affine smooth curves $X$, the group $\overline{H}^1(X_1, W^{\dagger}\Omega^{\bullet}_X)$ are not  finitely generated over $W$. 
On the other hand, we have the following proposition. 

\begin{proposition}
For an affine smooth curve $X$ over $k$ whose smooth compactification has genus $0$, 
$\overline{H}^1(X, W^{\dagger}\Omega^{\bullet}_X)$ is finitely generated over $W$. 
\end{proposition}

\begin{proof}
For a finite extension $k'$ of $k$, we have the base change property 
$H^1(X, W^{\dagger}\Omega^{\bullet}_X) \otimes_W W(k') \cong 
H^1(X \otimes_k k', W^{\dagger}\Omega^{\bullet}_{X \otimes_k k'})$ and so 
$\overline{H}^1(X, W^{\dagger}\Omega^{\bullet}_X) \otimes_W W(k') \cong 
\overline{H}^1(X \otimes_k k', W^{\dagger}\Omega^{\bullet}_{X \otimes_k k'})$. 
Hence, if $\overline{H}^1(X \otimes_k k', W^{\dagger}\Omega^{\bullet}_{X \otimes_k k'})$
is finitely generated over $W(k')$, $\overline{H}^1(X, W^{\dagger}\Omega^{\bullet}_X)$ 
is finitely generated over $W$. Thus we may replace $k$ by a finite extension of it 
to prove the proposition. Thus we may assume that 
$X = {\mathbf{A}}^1_k \setminus \{\overline{\alpha}_1, \dots, \overline{\alpha}_r\}$ 
for some distinct $\overline{\alpha}_i \in k$ $(1 \leq i \leq r)$. 

We prove the finiteness of $\overline{H}^1(X, W^{\dagger}\Omega^{\bullet}_X)$ 
by induction on $r$. If $r=0$, $X = {\mathbf{A}}^1_k$. In this case, 
$\overline{H}^1(X, W^{\dagger}\Omega^{\bullet}_X) \subset H^1_{\rm rig}(X/K) = \{0\}$
and so the claim is true. 

If $r=1$, we may assume that $X = {\mathbf{A}}^1_k \setminus \{0\}$. Then 
$\overline{H}^1(X, W^{\dagger}\Omega^{\bullet}_X) \cong \overline{H}^1_{\rm MW}(X/W) 
= \overline{H}^1(\Omega^{\bullet}_{A^{\dagger}})$, where 
\[ A^{\dagger} = \{ \sum_{n \in {\mathbf{Z}}} a_n x^n \,|\, a_n \in W \text{ and  } \exists \epsilon > 0, \exists C \in {\mathbf{R}}, \forall n \in {\mathbf{Z}}, 
\nu_p(a_n) \geq \epsilon |n| + C\} \] 
is the weak completion of $W[x,x^{-1}]$. 
Then any element of $\overline{H}^1(\Omega^{\bullet}_{A^{\dagger}})$ is written in the form 
$
\left[ (\sum_{n \in {\mathbf{Z}}} a_n x^n)dx \right]  
$ with $\sum_{n \in {\mathbf{Z}}} a_n x^n \in A^{\dagger}$, 
and it can be rewritten as 
\[ a_{-1}[x^{-1}dx] + \left[ (\sum_{n \not= -1} a_n x^n)dx \right] = 
a_{-1}[x^{-1}dx] +  \left[d(\sum_{n \not= -1} \frac{a_n}{n+1} x^{n+1})\right] = 
a_{-1}[x^{-1}dx] \in W[x^{-1}dx], \]
because $\sum_{n \not= -1} \frac{a_n}{n+1} x^{n+1} \in A^{\dagger} \otimes \Q$.  
Thus $\overline{H}^1(\Omega^{\bullet}_{A^{\dagger}}) = W[x^{-1}dx]$ and it 
is finitely generated over $W$, as required. 

If $r \geq 2$, we set $X_1 := {\mathbf{A}}^1_k \setminus \{\overline{\alpha}_1, \dots, \overline{\alpha}_{r-1}\}$, 
$X_2 := {\mathbf{A}}^1_k \setminus \{\overline{\alpha}_{r}\}$. 
Take lifts $\alpha_i \in W$ of $\overline{\alpha}_i$ for $1 \leq i \leq r$. 
Then, for $i=1,2$, 
$\overline{H}^1(X_i, W^{\dagger}\Omega^{\bullet}_X) \cong \overline{H}^1_{\rm MW}(X_i/W) 
= \overline{H}^1(\Omega^{\bullet}_{A_i^{\dagger}})$, where 
$A_1^{\dagger}$ is the weak completion of 
$W[x,(x-\alpha_1)^{-1}, \dots, (x - \alpha_{r-1})^{-1}]$ and 
$A_2^{\dagger}$ is the weak completion of 
$W[x,(x-\alpha_r)^{-1}]$. Also, 
$\overline{H}^1(X, W^{\dagger}\Omega^{\bullet}_X) \cong \overline{H}^1_{\rm MW}(X/W) 
= \overline{H}^1(\Omega^{\bullet}_{A^{\dagger}})$, where 
$A^{\dagger}$ is the weak completion of 
$W[x,(x-\alpha_1)^{-1}, \dots, (x - \alpha_{r})^{-1}]$. 
Then, as a special case of \cite[Lemma 7]{M_1972}, the canonical map 
$A_1^{\dagger} \oplus A_2^{\dagger} \to A^{\dagger}$ is surjective, and so 
the map $\Omega^1_{A_1^{\dagger}} \oplus \Omega^1_{A_2^{\dagger}} \to 
\Omega^1_{A^{\dagger}}$ is also surjective. Thus we see that 
the map $\overline{H}^1(\Omega^{\bullet}_{A_1^{\dagger}}) \oplus 
\overline{H}^1(\Omega^{\bullet}_{A_2^{\dagger}}) \to 
\overline{H}^1(\Omega^{\bullet}_{A^{\dagger}})$ is surjective. 
Because $\overline{H}^1(\Omega^{\bullet}_{A_1^{\dagger}}) \oplus 
\overline{H}^1(\Omega^{\bullet}_{A_2^{\dagger}})$ is finitely generated over $W$ 
by induction hypothesis, so is $\overline{H}^1(\Omega^{\bullet}_{A^{\dagger}})$, 
as required. 
\end{proof}

Consequently, on the finiteness of $\overline{H}^1(X_1, W^{\dagger}\Omega^{\bullet}_X)$ for an 
affine smooth curve $X$, we conjecture the following.

\begin{conjecture}
Let $X$ be an affine smooth curve over $k$. 
Then $\overline{H}^1(X_1, W^{\dagger}\Omega^{\bullet}_X)$ is finitely generated over $W$ if and only if 
its smooth compactification has genus $0$. 
\end{conjecture}

For higher dimensional varieties, 
we have the following as a simple consequence of 
Theorem \ref{main1}. 

\begin{theorem}\label{main4}
Let $X$ be a projective smooth variety over $k$ and let $a \geq 2$ be 
an integer prime to $p$. 
Then there exists a generically \'etale morphism $f: Y \to X$ of degree $a$ 
with $Y$ smooth such that $\overline{H}^1(Y, W^{\dagger}\Omega^{\bullet}_{Y})$ 
is not finitely generated over $W$. 
\end{theorem}

\begin{proof}
Take a closed embedding $X \subset {\mathbf{P}}_k^N$ into a projective space and 
take two transversal hyperplanes $H_1, H_2$ in ${\mathbf{P}}_k^N$ which meet 
$X$ smoothly and transversally. Let $A := H_1 \cap H_2 \cap X$ and let $b: \widetilde{X} \to X$ 
be the blow-up of $X$ along $A$. Then we have canonically a pencil structure 
$g: \widetilde{X} \to {\mathbf{P}}_k^1$. 

Let $h: C \to {\mathbf{A}}_k^1 \hookrightarrow {\mathbf{P}}_k^1$
be the superelliptic curve $y^a = x^b + 1$ (where $b$ is a positive integer coprime to 
$ap$ and $x$ is the coordinate of ${\mathbf{A}}_k^1$), and let 
$C' \subset C$ be the open subscheme on which $h$ is \'etale. 
Now let $g': Y \to C'$ be the pullback of  $g: \widetilde{X} \to {\mathbf{P}}_k^1$ by the 
composite $C' \hookrightarrow C \xrightarrow{h} {\mathbf{P}}_k^1$. 
Then we have the canoncical map $Y \to \widetilde{X} \to X$ which 
is generically \'etale of degree $a$, and $Y$ is smooth. 

Take a finite extension $k'$ of $k$ such that $A$ admits a $k'$-rational point 
$t$. In the following,  for a scheme or a morphism of schemes $?$, we denote $? \otimes_k k'$ simply by $?_{k'}$. 
Then $b_{k'}^{-1}(t) \cong {\mathbf{P}}^1_{k'}$ and so it defines a section $s: {\mathbf{P}}^1_{k'} \to \widetilde{X}_{k'}$ of 
$g_{k'}: \widetilde{X}_{k'} \to {\mathbf{P}}_{k'}^1$, and it induces a section $s': C'_{k'} \to Y_{k'}$ of 
the map $g'_{k'}: Y_{k'} \to C'_{k'}$. 

Thus we have maps between cohomology groups modulo torsion 
\begin{equation}\label{eq:revise2}
\overline{H}^1(C'_{k'}, W^{\dagger}\Omega^{\bullet}_{C'_{k'}}) \xrightarrow{{g'}_{k'}^*} 
\overline{H}^1(Y_{k'}, W^{\dagger}\Omega^{\bullet}_{Y_{k'}}) \xrightarrow{{s'}^*} 
\overline{H}^1(C'_{k'}, W^{\dagger}\Omega^{\bullet}_{C'_{k'}}) 
\end{equation}
whose composition is the identity. So the map ${g'}_{k'}^*$ is injective. 
By Theorem \ref{main1} and Proposition \ref{main3}, 
$\overline{H}^1(C'_{k'}, W^{\dagger}\Omega^{\bullet}_{C'_{k'}})$ is not finitely generated over 
$W(k')$. Hence $\overline{H}^1(Y_{k'}, W^{\dagger}\Omega^{\bullet}_{Y_{k'}})$ is not 
finitely generated over $W(k')$ either. Then, since 
\[ \overline{H}^1(Y_{k'}, W^{\dagger}\Omega^{\bullet}_{Y_{k'}}) \cong 
\overline{H}^1(Y, W^{\dagger}\Omega^{\bullet}_{Y}) \otimes_W W(k'), \] 
we conclude that $\overline{H}^1(Y, W^{\dagger}\Omega^{\bullet}_{Y})$ is not finitely generated 
over $W$. So the proof is finished. 
\end{proof}

\begin{remark}\label{main5}
We can construct examples of non-affine smooth varieties 
$X$ such that $\overline{H}^1(X, W^{\dagger}\Omega^{\bullet}_{X})$ is not 
finitely generated over $W$ in the following two ways. 

First, let $k$ be a perfect field of characteristic $p>0$, 
let $a, b \geq 2$ be coprime integers prime to $p$ and let 
$C$ be the superelliptic curve 
$x^a + y^b + 1 = 0$. Let $k'$ be a finite extension of $k$ such that 
$C$ has two distinct $k'$-rational closed points $s, t$. Then, if we set 
$X := (C \times_k C) \,\setminus\, (s \times_k s)$, it is a non-affine smooth variety over $k$. 
On the other hand, we have morphisms 
\[ \iota:C_{k'} \cong C \times_k t \hookrightarrow X, \quad 
\pi: X \hookrightarrow C \times_k C \xrightarrow{\text{$1$st proj.}} C \] 
such that $\iota \circ \pi: C_{k'} \to C$ is the natural projection. 
Then we have maps between cohomology groups modulo torsion 
\[ \overline{H}^1(C, W^{\dagger}\Omega^{\bullet}_{C}) \xrightarrow{\pi^*} 
\overline{H}^1(X, W^{\dagger}\Omega^{\bullet}_{X}) \xrightarrow{\iota^*} 
\overline{H}^1(C_{k'}, W^{\dagger}\Omega^{\bullet}_{C_{k'}}) \] 
whose composition is the injection 
\[ \overline{H}^1(C, W^{\dagger}\Omega^{\bullet}_{C}) 
\hookrightarrow 
\overline{H}^1(C, W^{\dagger}\Omega^{\bullet}_{C}) 
\otimes_{W} W(k') \cong  \overline{H}^1(C_{k'}, W^{\dagger}\Omega^{\bullet}_{C_{k'}}). \] 
Thus $\pi^*$ is injective.  
By Theorem \ref{main1}, 
$\overline{H}^1(C, W^{\dagger}\Omega^{\bullet}_{C})$ is not finitely generated over 
$W$. Hence $\overline{H}^1(X, W^{\dagger}\Omega^{\bullet}_{X})$ is not 
finitely generated over $W$ either. 

Second, consider the smooth variety $Y$ in the proof of Theorem \ref{main4} 
such that $\dim Y \geq 2$. 
If $Y$ is not affine, this gives an example we want. Otherwise, 
take a closed point $t$ of $Y$ whose inverse image in $Y_{k'}$ does not meet 
the image of $s': C'_{k'} \to Y_{k'}$ in the proof of Theorem \ref{main4}. 
Then the diagram \eqref{eq:revise2} remains to exist if we replace $Y$ by 
$Y' := Y \setminus t$. Then $Y'$ is a non-affine smooth variety such that 
$\overline{H}^1(Y', W^{\dagger}\Omega^{\bullet}_{Y'})$ is not 
finitely generated over $W$, as required. 
\end{remark}

\end{document}